\documentclass[11pt]{amsart}

\usepackage[T1]{fontenc}
\usepackage[a4paper,margin=1in]{geometry}
\usepackage{amsmath,amssymb,mathtools}
\usepackage{enumitem}
\usepackage{microtype}
\usepackage[colorlinks=true, linkcolor=blue,urlcolor=blue]{hyperref}
\hypersetup{
  pdftitle={Factorization Bounds and Irreducibility Criteria in Hurwitz Series Rings},
  pdfauthor={Morteza Ahmadi},
  pdfkeywords={Hurwitz series ring, irreducibility, factorization length, discrete valuation domain, Newton polygon}
}

\allowdisplaybreaks

\newtheorem{theorem}{Theorem}[section]
\newtheorem{lemma}[theorem]{Lemma}
\newtheorem{proposition}[theorem]{Proposition}
\newtheorem{corollary}[theorem]{Corollary}
\theoremstyle{definition}
\newtheorem{definition}[theorem]{Definition}
\newtheorem{example}[theorem]{Example}
\theoremstyle{remark}
\newtheorem{remark}[theorem]{Remark}

\newcommand{\N}{\mathbb N}
\newcommand{\Z}{\mathbb Z}
\newcommand{\Q}{\mathbb Q}
\newcommand{\supp}{\operatorname{supp}}
\newcommand{\lengthset}{\mathsf L_H}
\newcommand{\NP}{\operatorname{NP}_H}
\newcommand{\U}{\mathrm U}

\title[Factorization in Hurwitz Series Rings]{Factorization Bounds and Irreducibility Criteria in Hurwitz Series Rings}
\author{Morteza Ahmadi}
\address{Department of Pure Mathematics, Faculty of Mathematical Sciences\\
	Tarbiat Modares University, P.O.Box:14115-134, Tehran, Iran}
\email{morteza.ahmadi23@gmail.com\\ mortezy.ahmadi@modares.ac.ir}
\date{}

\begin{document}

\begin{abstract}
Let $R$ be a principal ideal domain and let $HR$ denote the Hurwitz series ring over $R$.  We first give a recursive splitting formula for a Hurwitz series whose constant coefficient is a product of two coprime nonunits.  This construction leads to irreducibility tests for prime-power constant coefficients, decompositions indexed by the distinct prime divisors of the constant coefficient, and upper and lower bounds for the length of every irreducible factorization.  Over a discrete valuation domain, the valuation of any selected coefficient yields a sharper length bound.  We then introduce the Hurwitz--Newton polygon and prove a product rule on ranges in which the relevant binomial coefficients are units.  Primitive one-edge polygons consequently provide Dumas-type irreducibility criteria.  Finally, localization is used to combine independent criteria at several prime elements.  Examples over $\Z$, localized polynomial rings, and the Gaussian integers illustrate the results.
\end{abstract}

\maketitle

\section{Introduction}

Hurwitz series rings arise naturally in differential algebra and provide a multiplication law governed by binomial convolution \cite{Keigher1997}.  Their ring-theoretic properties have been studied from several complementary viewpoints.  In particular, skew Hurwitz serieswise Armendariz rings were investigated in \cite{AhmadiMoussaviNourozi2014}, the behavior of nilradicals was analyzed in \cite{AhmadiMoussaviNourozi2015}, nilpotent elements of skew Hurwitz polynomial rings were considered in \cite{NouroziMoussaviAhmadi2017}, singular ideals were characterized in \cite{Ahmadi2019}, and McCoy-type properties were developed in \cite{NouroziRahmatiAhmadi2021}.  The notation used below follows the Hurwitz-series conventions of \cite{Ahmadi2019}.

This paper studies factorization directly inside the commutative Hurwitz series ring.  The first group of results uses only the constant coefficient and the binomial convolution law.  A coprime factorization of the constant coefficient can be lifted recursively to a factorization of the whole series.  In the opposite direction, a prime-power constant coefficient, together with a suitable higher coefficient, restricts or prevents nontrivial factorizations.  These arguments give quantitative bounds for all irreducible factorization lengths without requiring the Hurwitz series ring itself to be a unique factorization domain.

A second group of results uses valuations.  The binomial coefficients in the Hurwitz product can affect a valuation, so the Newton-polygon arguments are stated on a range where all relevant factorials are units.  On such a range, a product rule for Hurwitz--Newton polygons yields primitive-edge irreducibility tests.  Local criteria at distinct prime elements can then be assembled by a coprime splitting and localization argument.

The paper is organized as follows.  Section~\ref{sec:preliminaries} fixes the notation and records basic coefficient formulas.  Section~\ref{sec:constant} develops constant-coefficient splitting, irreducibility criteria, and factorization-length bounds.  Section~\ref{sec:newton} establishes the Hurwitz--Newton polygon criteria.  Section~\ref{sec:localization} treats localization and several prime divisors. %, and Section~\ref{sec:examples} gives applications.

\section{Hurwitz series notation and preliminary facts}\label{sec:preliminaries}

Throughout, $\N=\{0,1,2,\ldots\}$.  Let $R$ be a ring with identity and let $\alpha\colon R\to R$ be an endomorphism.  Following \cite{Ahmadi2019}, the skew Hurwitz series ring $(H(R),\alpha)$, also denoted by $(HR,\alpha)$, consists of all functions $f\colon\N\to R$.  Addition is componentwise, and multiplication is defined by
\begin{equation}\label{eq:skew-product}
 (fg)(n)=\sum_{k=0}^{n}\binom{n}{k}f(k)\alpha^k\bigl(g(n-k)\bigr),
 \qquad n\in\N.
\end{equation}
For $f\in(HR,\alpha)$, set
\[
 \supp(f)=\{i\in\N:f(i)\ne0\}.
\]
When $f\ne0$, let $\Pi(f)$ be the least element of $\supp(f)$, and let $\Delta(f)$ be its greatest element whenever that element exists.

For $n\ge1$, define $h_n\colon\N\to R$ by
\[
 h_n(n-1)=1,
 \qquad
 h_n(m)=0\quad(m\ne n-1).
\]
For $r\in R$, let $h'_r\colon\N\to R$ satisfy $h'_r(0)=r$ and $h'_r(n)=0$ for $n\ne0$.  The identity of $(HR,\alpha)$ is $h_1$, and the subring
\[
 R'=\{h'_r:r\in R\}
\]
is naturally isomorphic to $R$.  The skew Hurwitz polynomial ring $(hR,\alpha)$ is the subring formed by the elements having finite support, equivalently those nonzero $f$ for which $\Delta(f)<\infty$.

The factorization results in this paper concern a commutative ring $R$ and the identity endomorphism.  We therefore write $HR$ and $hR$ for $(HR,\operatorname{id}_R)$ and $(hR,\operatorname{id}_R)$, respectively.  If $a_i=f(i)$, then
\[
 f=\sum_{i=0}^{\infty}h'_{a_i}h_{i+1},
\]
and \eqref{eq:skew-product} becomes
\begin{equation}\label{eq:hurwitz-product}
 (fg)(n)=\sum_{k=0}^{n}\binom{n}{k}f(k)g(n-k).
\end{equation}
In particular,
\begin{equation}\label{eq:basis-product}
 h_{i+1}h_{j+1}=\binom{i+j}{i}h_{i+j+1}
 \qquad(i,j\in\N).
\end{equation}

\begin{lemma}[Unit criterion]\label{lem:unit}
Let $R$ be a commutative ring.  Then $f\in HR$ is invertible if and only if $f(0)\in\U(R)$.
\end{lemma}

\begin{proof}
If $fg=h_1$, evaluation at $0$ gives $f(0)g(0)=1$.  Conversely, suppose that $f(0)$ is a unit.  Define $g(0)=f(0)^{-1}$ and, for $n\ge1$, set
\[
 g(n)=-f(0)^{-1}\sum_{k=1}^{n}\binom{n}{k}f(k)g(n-k).
\]
The coefficient formula \eqref{eq:hurwitz-product} then gives $(fg)(0)=1$ and $(fg)(n)=0$ for every $n\ge1$.  Hence $fg=h_1$.
\end{proof}

For later coefficient estimates, let $f_t\in HR$ for $1\le t\le s$.  Repeated application of \eqref{eq:hurwitz-product} gives
\begin{equation}\label{eq:multifactor}
 (f_1\cdots f_s)(n)
 =\sum_{i_1+\cdots+i_s=n}
 \binom{n}{i_1,\ldots,i_s}
 \prod_{t=1}^{s}f_t(i_t),
\end{equation}
where
\[
 \binom{n}{i_1,\ldots,i_s}=\frac{n!}{i_1!\cdots i_s!}.
\]

Let $R$ be a unique factorization domain.  For a nonzero nonunit $a\in R$, write $\omega(a)$ for the number of pairwise nonassociate prime divisors of $a$, and $\Omega(a)$ for the number of prime factors counted with multiplicity.  Since $HR$ need not have unique factorization, we use
\[
 \lengthset(f)=\{s\ge1~|~f=q_1\cdots q_s
 \text{ with every }q_i\text{ irreducible in }HR\}.
\]
All length estimates below are assertions about the elements of this set; they do not presuppose that all irreducible factorizations have the same length.

\section{Constant coefficients and factorization lengths}\label{sec:constant}

\begin{theorem}[Coprime lifting]\label{thm:coprime-lifting}
Let $R$ be a principal ideal domain and let $f\in HR$ have a nonzero nonunit constant coefficient.  Assume
\[
 f(0)=mn,
 \qquad (m,n)=R,
\]
where both $m$ and $n$ are nonunits.  Then $f$ is reducible in $HR$.
More precisely, if $u,v\in R$ satisfy $mu+nv=1$, there is $g\in HR$ with $g(0)=0$ such that
\begin{equation}\label{eq:coprime-factorization}
 f=(h'_m+h'_vg)(h'_n+h'_ug).
\end{equation}
\end{theorem}

\begin{proof}
Set $g(0)=0$.  Once $g(1),\ldots,g(t-1)$ have been chosen, define
\begin{equation}\label{eq:g-recursion}
 g(t)=f(t)-uv\sum_{i=1}^{t-1}\binom{t}{i}g(i)g(t-i)
 \qquad(t\ge1).
\end{equation}
The constant coefficient on the right-hand side of \eqref{eq:coprime-factorization} is $mn=f(0)$.  For $t\ge1$, its $t$-th coefficient equals
\[
 (mu+nv)g(t)+uv(g^2)(t)
 =g(t)+uv\sum_{i=1}^{t-1}\binom{t}{i}g(i)g(t-i),
\]
which is $f(t)$ by \eqref{eq:g-recursion}.  The two factors in \eqref{eq:coprime-factorization} have nonunit constant coefficients $m$ and $n$, so Lemma~\ref{lem:unit} shows that both are nonunits.
\end{proof}

\begin{theorem}[Prime-power test]\label{thm:prime-power-test}
Let $R$ be a principal ideal domain and let $f\in HR$.  Suppose that $f(0)$ is associate to $p^k$, where $p$ is prime and $k\ge1$.  If $k=1$, or if $p\nmid f(1)$, then $f$ is irreducible in $HR$.
\end{theorem}

\begin{proof}
If $k=1$ and $f=gh$, the equality $f(0)=g(0)h(0)$ forces one of $g(0)$ and $h(0)$ to be a unit.  The corresponding series is a unit by Lemma~\ref{lem:unit}.

Now assume $k>1$ and suppose that $f=gh$ with $g$ and $h$ both nonunits.  Their constant coefficients are then divisible by $p$.  Since
\[
 f(1)=g(0)h(1)+g(1)h(0),
\]
we obtain $p\mid f(1)$, contrary to the hypothesis.
\end{proof}

\begin{theorem}[Decomposition by distinct prime divisors]\label{thm:prime-decomposition}
Let $R$ be a principal ideal domain and let $f\in HR$ satisfy
\[
 f(0)=u\prod_{i=1}^{r}p_i^{k_i},
\]
where $u$ is a unit, the $p_i$ are pairwise nonassociate primes, and $k_i\ge1$.  If $(f(0),f(1))=R$, then
\[
 f=g_1\cdots g_r
\]
for irreducible series $g_i\in HR$ such that $g_i(0)$ is associate to $p_i^{k_i}$.
\end{theorem}

\begin{proof}
Apply Theorem~\ref{thm:coprime-lifting} repeatedly to the pairwise coprime prime-power components of $f(0)$.  This gives a factorization $f=g_1\cdots g_r$ with $g_i(0)\sim p_i^{k_i}$.  For the coefficient of index $1$, formula \eqref{eq:multifactor} reduces to
\[
 f(1)=\sum_{i=1}^{r}g_i(1)\prod_{j\ne i}g_j(0).
\]
Fix $i$ and reduce this equality modulo $p_i$.  Every summand except the $i$-th is divisible by $p_i$, whereas $\prod_{j\ne i}g_j(0)$ is not.  Because $p_i\nmid f(1)$, it follows that $p_i\nmid g_i(1)$.  Theorem~\ref{thm:prime-power-test} now proves that each $f_i$ is irreducible.
\end{proof}

\begin{theorem}[Bounds from the constant coefficient]\label{thm:constant-bounds}
Let $R$ be a principal ideal domain, let $f\in HR$, and assume that $f(0)$ is a nonzero nonunit.  For every $s\in\lengthset(f)$,
\[
 \omega(f(0))\le s\le\Omega(f(0)).
\]
Consequently, a series with prime constant coefficient is irreducible.
\end{theorem}

\begin{proof}
Take an irreducible factorization $f=g_1\cdots g_s$.  Since each $g_i$ is a nonunit, Lemma~\ref{lem:unit} implies that every $g_i(0)$ is a nonunit.  The equality
\[
 f(0)=g_1(0)\cdots g_s(0)
\]
therefore yields $s\le\Omega(f(0))$.

If some $g_i(0)$ had two nonassociate prime divisors, its prime-power factors could be divided into two coprime nonunit groups.  Theorem~\ref{thm:coprime-lifting}, applied to $g_i$, would then contradict the irreducibility of $g_i$.  Hence every $g_i(0)$ is associate to a power of a single prime.  At least $\omega(f(0))$ such factors are required to account for all distinct prime divisors of $f(0)$.  The final assertion also follows directly from Theorem~\ref{thm:prime-power-test} with $k=1$.
\end{proof}
\begin{example}%[The upper length bound is attained]\label{ex:upper-bound}
	Let $p$ be prime.  Then
	\(
	f=h'_p+h_2
	\)
	is irreducible in $H(\Z)$ by Theorem~\ref{thm:constant-bounds}.  Therefore, for $n\ge2$,
	\(
	g=f^n
	\)
	has an irreducible factorization of length $n$.  Since $g(0)=p^n$, Theorem~\ref{thm:constant-bounds} also gives the upper bound $n$, so the displayed factorization has maximal possible length.
\end{example}

\begin{corollary}[Squarefree constant coefficient]\label{cor:squarefree}
Under the assumptions of Theorem~\ref{thm:constant-bounds}, if $f(0)$ is squarefree, then
\[
 \lengthset(f)\subseteq\{\omega(f(0))\}.
\]
Thus every irreducible factorization, whenever one exists, has exactly $\omega(f(0))$ factors.
\end{corollary}

\begin{proof}
For a squarefree element, $\omega(f(0))=\Omega(f(0))$.  The result follows immediately from Theorem~\ref{thm:constant-bounds}.
\end{proof}

\begin{theorem}[A selected coefficient over $\Z$]\label{thm:selected-Z}
Let $f\in H(\Z)$ and suppose that $f(0)=\pm p^k$, where $p$ is prime and $k\ge1$.  If $p\nmid f(j)$ for some $j\ge1$, then every $s\in\lengthset(f)$ satisfies
\[
 s\le\min\{k,j\}.
\]
In particular, $f$ is irreducible when $k=1$ or $j=1$.
\end{theorem}

\begin{proof}
Let $f=g_1\cdots g_s$ be an irreducible factorization.  Each $g_t(0)$ is divisible by $p$, and comparison of constant coefficients gives $s\le k$.  Suppose that $s>j$.  For every $s$-tuple $(i_1,\ldots,i_s)$ of nonnegative integers whose sum is $j$, at least one $i_t$ equals $0$.  The corresponding term in \eqref{eq:multifactor} contains $g_t(0)$ and is therefore divisible by $p$.  Hence $p\mid f(j)$, a contradiction.  Thus $s\le j$ as well.  When $k=1$ or $j=1$, irreducibility follows directly from Theorem~\ref{thm:prime-power-test}.
\end{proof}
\begin{example}%[A bound determined by a higher coefficient]\label{ex:higher-coefficient}
	Let $p$ be prime, let $1\le j<k$, and let $f\in H(\Z)$.  Then
	\[
	g=h'_{p^k}\pm h_{j+1}+h_{j+2}f
	\]
	has $g(j)=\pm1$.  By Theorem~\ref{thm:selected-Z}, every irreducible factorization of $g$ has at most $j$ factors.  When $j=1$, the same theorem gives irreducibility.
\end{example}

\begin{theorem}[A selected coefficient over a DVR]\label{thm:selected-DVR}
Let $(R,v)$ be a discrete valuation domain with uniformizer $\pi$, with the convention $v(0)=\infty$.  Let $f\in HR$ satisfy
\[
 f(0)=u\pi^k,
 \qquad u\in\U(R),\quad k\ge1.
\]
If $v(f(j))=\ell<\infty$ for some $j\ge1$, then every $s\in\lengthset(f)$ satisfies
\[
 s\le\min\{k,j+\ell\}.
\]
\end{theorem}

\begin{proof}
Write an irreducible factorization as $f=g_1\cdots g_s$.  Since $g_t(0)$ is a nonunit, write
\[
 g_t(0)=u_t\pi^{k_t},
 \qquad k_t\ge1.
\]
The constant coefficient gives $k_1+\cdots+k_s=k$, so $s\le k$.  In each summand of \eqref{eq:multifactor} contributing to $f(j)$, at least $s-j$ of the indices are zero.  The product of the associated constant coefficients has valuation at least $s-j$, while all remaining coefficients and the multinomial coefficient have nonnegative valuation.  Therefore
\[
 v(f(j))\ge s-j.
\]
If $s>j+\ell$, this inequality contradicts $v(f(j))=\ell$.
\end{proof}
\begin{example}%[A valuation bound over $\Z_{(p)}$]\label{ex:dvr-bound}
	Let $p$ be prime, $k\ge2$, and $f\in H(\Z_{(p)})$.  Consider
	\[
	\begin{aligned}
		g={}&h'_{p^k/(p+1)}
		+h'_p h_2+h'_p h_3+\cdots+h'_p h_k
		+h'_{p+1}h_{k+1}+h_{k+2}f.
	\end{aligned}
	\]
	Its constant coefficient has $p$-adic valuation $k$, whereas $v_p(g(1))=1$.  Theorem~\ref{thm:selected-DVR} shows that every irreducible factorization of $g$ has length at most $2$.
\end{example}

\begin{corollary}[Valuation envelope]\label{cor:valuation-envelope}
In the setting of Theorem~\ref{thm:selected-DVR}, every $s\in\lengthset(f)$ satisfies
\[
 s\le \min\bigl(\{k\}\cup\{j+v(f(j))~|~j\ge1\}\bigr).
\]
\end{corollary}

\section{Hurwitz--Newton polygons}\label{sec:newton}

Let $(R,v)$ be a discrete valuation domain.  The binomial coefficients in \eqref{eq:hurwitz-product} need not be units, and their valuations must be controlled before a polygonal product argument can be used.

\begin{definition}[Binomial-unit range]\label{def:binomial-range}
An integer $N\ge1$ is called a binomial-unit range for $v$ if
\[
 v(t!)=0\qquad(0\le t\le N).
\]
Equivalently, every binomial or multinomial coefficient whose upper index is at most $N$ is a unit of $R$.
\end{definition}

For a nonzero $f\in HR$, the Hurwitz--Newton polygon $\NP(f)$ is the lower convex polygon determined by the points
\[
 (i,v(f(i)))\qquad\bigl(i\in\supp(f)\bigr).
\]
Only the negative-slope part will be needed below.

\begin{proposition}[Product rule in a binomial-unit range]\label{prop:polygon-product}
Let $N$ be a binomial-unit range for $v$, let $f,g\in HR$, and put $F=fg$.  Suppose that the negative-slope part of $\NP(F)$ terminates at a point $(n,0)$ with $n\le N$.  Through abscissa $n$, its edges are obtained by combining the negative-slope edges of $\NP(f)$ and $\NP(g)$ according to the Newton-polygon product rule.  In particular, equal-slope edges contribute additively in both horizontal length and vertical drop.
\end{proposition}

\begin{proof}
Introduce the degree-$N$ polynomial associated with $f$ by
\[
 \Phi_N(f)=\sum_{i=0}^{N}\frac{f(i)}{i!}X^i
 \in R[X]/(X^{N+1}).
\]
All denominators are units because $N$ is a binomial-unit range.  For $n\le N$, the coefficient of $X^n$ in $\Phi_N(f)\Phi_N(g)$ is
\[
 \sum_{i=0}^{n}\frac{f(i)g(n-i)}{i!(n-i)!}
 =\frac{1}{n!}\sum_{i=0}^{n}\binom{n}{i}f(i)g(n-i)
 =\frac{(fg)(n)}{n!}.
\]
Thus
\[
 \Phi_N(fg)=\Phi_N(f)\Phi_N(g)
 \quad\text{in }R[X]/(X^{N+1}).
\]
Moreover, $v(f(i)/i!)=v(f(i))$ for $i\le N$.  Hence the relevant polygons are unchanged by $\Phi_N$, and the polynomial product theorem for Newton polygons applies; see, for example, Dumas \cite{Dumas1906}.
\end{proof}

\begin{theorem}[Primitive-edge irreducibility]\label{thm:primitive-edge}
Let $(R,v)$ be a discrete valuation domain with uniformizer $\pi$, and let $f\in HR$ satisfy
\[
 f(0)=u\pi^k,
 \qquad u\in\U(R),\quad k\ge1.
\]
Assume that an integer $n\ge1$ has the following properties:
\begin{enumerate}[label=\textup{(\roman*)}]
\item $n$ lies in a binomial-unit range for $v$;
\item $v(f(n))=0$ and $\gcd(k,n)=1$;
\item for $1\le i<n$,
\[
 v(f(i))>\frac{k(n-i)}{n}.
\]
\end{enumerate}
Then $f$ is irreducible in $HR$.
\end{theorem}

\begin{proof}
The stated inequalities place every point with abscissa strictly between $0$ and $n$ above the segment joining $(0,k)$ to $(n,0)$.  Thus the negative part of $\NP(f)$ is the single edge of slope $-k/n$.

Suppose that $f=g_1g_2$ with $g_1$ and $g_2$ both nonunits.  Write
\[
 v(g_1(0))=\ell,
 \qquad
 v(g_2(0))=k-\ell,
 \qquad 0<\ell<k.
\]
Neither factor can have all its coefficients divisible by $\pi$, because $f(n)$ is a unit.  Proposition~\ref{prop:polygon-product} therefore supplies positive integers $r$ and $s$ with $r+s=n$ such that the negative edges contributed by $g_1$ and $g_2$ have the common slope $-k/n$.  Consequently,
\[
 \frac{\ell}{r}=\frac{k-\ell}{s}=\frac{k}{n},
\]
so $\ell n=kr$.  Since $\gcd(k,n)=1$, this implies $n\mid r$, which is impossible for $0<r<n$.  Hence no nontrivial factorization exists.
\end{proof}

The strict inequalities in Theorem~\ref{thm:primitive-edge} may equivalently be written as
\begin{equation}\label{eq:ceiling-bound}
 v(f(i))\ge \left\lfloor\frac{k(n-i)}{n}\right\rfloor+1
 \qquad(1\le i<n).
\end{equation}
\begin{example}%[A localized polynomial coefficient ring]\label{ex:localized-polynomial}
	Let $V=\Q[y]_{(1+y)}$.  This is a discrete valuation domain with uniformizer $1+y$.  Then
	\[
	f=h'_{(1+y)^8}+h'_{(1+y)^6}h_2
	+h'_{(1+y)^3}h_3+h'_y h_4
	\]
	has valuation points
	\[
	(0,8),\quad(1,6),\quad(2,3),\quad(3,0).
	\]
	The two interior points lie strictly above the segment from $(0,8)$ to $(3,0)$, and $\gcd(8,3)=1$.  Every integer factorial is a unit in $V$, so Theorem~\ref{thm:primitive-edge} proves that $f$ is irreducible in $H(V)$.
\end{example}

\begin{theorem}[Block-valuation criterion]\label{thm:block-criterion}
Let $f\in H(\Z)$ satisfy $f(0)=\pm p^k$, where $p$ is prime and $k\ge2$.  Suppose that $m\ge1$, $p>km+1$, and
\[
 p^{\ell}\mid f((k-\ell)m+i),
 \qquad
 p^{\ell+1}\nmid f((k-\ell)m+i)
\]
for every $1\le\ell\le k$ and $1\le i\le m$.  If $p\nmid f(km+1)$, then $f$ is irreducible in $H(\Z)$.
\end{theorem}

\begin{proof}
Use the $p$-adic valuation on $\Z_{(p)}$.  The inequality $p>km+1$ makes $km+1$ a binomial-unit range.  For $t=(k-\ell)m+i$, the ordinate of that segment is
\[
 \frac{k(km+1-t)}{km+1}
 =\frac{k(\ell m+1-i)}{km+1}<\ell=v_p(f(t)),
\]
because $\ell(km+1)-k(\ell m+1-i)=k(i-1)+\ell>0$.  Thus all intermediate valuation points lie strictly above the segment.  Since $\gcd(k,km+1)=1$, Theorem~\ref{thm:primitive-edge} gives irreducibility in $H(\Z_{(p)})$.  Any nontrivial factorization in $H(\Z)$ would remain nontrivial after localization, so no such factorization can occur.
\end{proof}

\begin{example}%[A block-valuation family]\label{ex:block-family}
	Let $k\ge2$ and let $p>k+1$ be prime.  Put
	\[
	f=h_1+h_2+h_3+\cdots,
	\]
	and define
	\[
	g=h'_{\pm p^k}
	+h'_{p^k}h_2+h'_{p^{k-1}}h_3+\cdots+h'_p h_{k+1}
	+h_{k+2}f.
	\]
	The coefficient valuations in degrees $1,\ldots,k$ are $k,k-1,\ldots,1$, and the coefficient in degree $k+1$ is a $p$-adic unit.  Theorem~\ref{thm:block-criterion}, with $m=1$, proves that $g$ is irreducible in $H(\Z)$.
\end{example}
\begin{corollary}[Initial-divisibility criterion]\label{cor:initial-divisibility}
Let $f\in H(\Z)$ satisfy $f(0)=\pm p^k$, where $p$ is prime and $k\ge1$.  Assume that for some $n\ge1$,
\[
 p>n,
 \qquad \gcd(k,n)=1,
\]
\[
 p^k\mid f(i)\quad(1\le i<n),
 \qquad
 p\nmid f(n).
\]
Then $f$ is irreducible in $H(\Z)$.
\end{corollary}

\begin{proof}
The $p$-adic valuation of every coefficient with index $1\le i<n$ is at least $k$, which is strictly larger than $k(n-i)/n$.  Theorem~\ref{thm:primitive-edge} applies because $p>n$.
\end{proof}
\begin{example}%[Multiplication by a Hurwitz unit]\label{ex:unit-multiple}
	Let $p$ be prime and let $k,j\ge1$ with $\gcd(k,j)=1$ and $p>j$.  Then
	\[
	f=h_1+h_2+h_3+\cdots
	\]
	is a unit because $f(0)=1$.  Set
	\[
	g=(h'_{\pm p^k}+h_{j+1})f.
	\]
	For $1\le t<j$, the coefficient $g(t)$ is divisible by $p^k$, while
	\[
	g(j)\equiv1\pmod p.
	\]
	Corollary~\ref{cor:initial-divisibility} gives the irreducibility of $g$.  Since $f$ is a unit, $h'_{\pm p^k}+h_{j+1}$ is irreducible as well.
\end{example}

\begin{remark}\label{rem:binomial-warning}
The range condition is essential to the argument above.  When a residue characteristic divides one of the factorials up to the relevant index, some binomial coefficients in \eqref{eq:hurwitz-product} acquire positive valuation.  The polygon of a product may then contain valuation contributions that are absent from the coefficient polygons of the factors, and Proposition~\ref{prop:polygon-product} cannot be used without an additional correction term.
\end{remark}

\section{Localization and several prime divisors}\label{sec:localization}

\begin{lemma}[Localization reflects irreducibility]\label{lem:localization}
Let $R$ be a principal ideal domain, let $p$ be a prime element, and let $R_{(p)}$ be the localization at $(p)$.  Suppose that $f\in HR$ has constant coefficient associate to $p^k$ for some $k\ge1$.  If $f$ is irreducible in $H(R_{(p)})$, then it is irreducible in $HR$.
\end{lemma}

\begin{proof}
Assume that $f=gh$ in $HR$ with $g$ and $h$ nonunits.  Since $g(0)h(0)$ is associate to $p^k$, both constant coefficients are divisible by positive powers of $p$.  They remain nonunits in $R_{(p)}$, and Lemma~\ref{lem:unit} shows that $g$ and $h$ remain nonunits in $H(R_{(p)})$.  Thus the same equality would be a nontrivial factorization after localization.
\end{proof}

\begin{theorem}[Independent local edges]\label{thm:several-primes}
Let $R$ be a principal ideal domain and let $f\in HR$ satisfy
\[
 f(0)=u\prod_{i=1}^{r}p_i^{k_i},
 \qquad r\ge2,
\]
where $u$ is a unit, the $p_i$ are pairwise nonassociate primes, and $k_i\ge1$.  For each $i$, let $v_i$ be the normalized valuation on $R_{(p_i)}$.  Suppose that there is an integer $n_i\ge1$ such that
\begin{enumerate}[label=\textup{(\roman*)}]
\item $n_i$ lies in a binomial-unit range for $v_i$;
\item $v_i(f(n_i))=0$ and $\gcd(k_i,n_i)=1$;
\item for $1\le t<n_i$,
\[
 v_i(f(t))>\frac{k_i(n_i-t)}{n_i}.
\]
\end{enumerate}
Then
\[
 f=g_1\cdots g_r
\]
for irreducible series $g_i\in HR$ with $g_i(0)\sim p_i^{k_i}$.
\end{theorem}

\begin{proof}
Repeated application of Theorem~\ref{thm:coprime-lifting} yields
\[
 f=g_1\cdots g_r,
 \qquad g_i(0)\sim p_i^{k_i}.
\]
Fix $i$ and work in $R_{(p_i)}$.  For $j\ne i$, the constant coefficient $g_j(0)$ is a unit, and hence $g_j$ is a unit in $H(R_{(p_i)})$ by Lemma~\ref{lem:unit}.  Such factors have no negative Newton edge.  Proposition~\ref{prop:polygon-product} therefore shows that the unique primitive negative edge imposed on $f$ by the hypotheses is also the negative edge of $g_i$.  Theorem~\ref{thm:primitive-edge} makes $g_i$ irreducible in $H(R_{(p_i)})$, and Lemma~\ref{lem:localization} then gives irreducibility in $HR$.
\end{proof}

\begin{corollary}[Several initial-divisibility conditions]\label{cor:several-initial}
In the setting of Theorem~\ref{thm:several-primes}, suppose that for each $i$ there is $n_i\ge1$ such that
\[
 v_i(n_i!)=0,
 \qquad \gcd(k_i,n_i)=1,
\]
\[
 p_i^{k_i}\mid f(t)\quad(1\le t<n_i),
 \qquad
 p_i\nmid f(n_i).
\]
Then $f$ is a product of $r$ irreducible Hurwitz series whose constant coefficients are associated, respectively, to $p_i^{k_i}$.
\end{corollary}

\begin{proof}
The divisibility assumptions imply the strict inequalities in Theorem~\ref{thm:several-primes} for every $i$.
\end{proof}

\begin{example}%[Two prime elements in the Gaussian integers]\label{ex:gaussian}
	Let $R=\Z[i]$.  The elements $19$ and $2+i$ are nonassociate Gaussian primes, and the residue characteristics of their localizations are $19$ and $5$, respectively.  Let $u\in\U(R)$, let $f\in HR$, and assume $2\le k\le5$.  Define
	\[
	g=h'_{19^k(2+i)u}+h'_4h_k+h_{k+1}f.
	\]
	The coefficients below index $k-1$ vanish, and $4$ is divisible by neither $19$ nor $2+i$.  For the prime $19$, use $n_1=k-1$ and exponent $k$; for the prime $2+i$, use $n_2=k-1$ and exponent $1$.  Since $k-1\le4$, both indices lie in the required binomial-unit ranges, and the coprimality conditions hold.  Corollary~\ref{cor:several-initial} therefore gives a factorization of $g$ into two irreducible Hurwitz series with constant coefficients associated to $19^k$ and $2+i$.
\end{example}

\begin{corollary}[Several block-valuation conditions]\label{cor:several-blocks}
Retain the notation of Theorem~\ref{thm:several-primes}.  Suppose that for each $j$ there is $m_j\ge1$ such that $k_jm_j+1$ lies in a binomial-unit range for $v_j$ and
\[
 v_j\bigl(f((k_j-\ell)m_j+i)\bigr)=\ell
\]
for $1\le\ell\le k_j$ and $1\le i\le m_j$, while
\[
 v_j(f(k_jm_j+1))=0.
\]
Then $f$ factors into $r$ irreducible Hurwitz series, with the $j$-th factor having constant coefficient associate to $p_j^{k_j}$.
\end{corollary}

\begin{proof}
For each $j$, the indicated valuations determine a primitive edge from $(0,k_j)$ to $(k_jm_j+1,0)$.  Apply Theorem~\ref{thm:several-primes}.
\end{proof}

%\section{Examples}\label{sec:examples}

%\section{Conclusion}

%The Hurwitz convolution permits strong factorization information to be extracted from a small number of coefficients.  Coprime factors of the constant coefficient lift recursively to factors of the entire series, while prime-power constant coefficients impose restrictive divisibility conditions on every possible factorization.  Counting the zero indices in the multinomial coefficient formula gives general factorization-length bounds, and valuations refine these bounds over discrete valuation domains.

%The Hurwitz--Newton polygon provides a second mechanism.  On a binomial-unit range, factorial normalization preserves coefficient valuations and converts the Hurwitz convolution into a polynomial product through the relevant degree.  Primitive negative edges then force irreducibility, and localization allows independent edges at distinct prime elements to be combined into an explicit irreducible decomposition.  Outside a binomial-unit range, the valuations of the binomial coefficients must be incorporated into any further polygonal theory.

\end{document}